\documentclass[10pt]{amsart}
\usepackage{amsmath,amsfonts,amssymb}
\usepackage{enumerate}

\newtheorem{lemma}{Lemma}[section]

\newtheorem{proposition}[lemma]{Proposition}
\newtheorem{theorem}[lemma]{Theorem}
\newtheorem{corollary}[lemma]{Corollary}
\theoremstyle{definition}
\newtheorem{definition}[lemma]{Definition}

\theoremstyle{remark}

\numberwithin{equation}{section} \numberwithin{table}{section}

\begin{document}
\title[Configurations in measurable sets]{A note on configurations in sets of positive density which occur at all large scales}

\author{Ian D. Morris}
\address{Department of Mathematics, University of Surrey, Guildford GU2 7XH, United Kingdom.}
\email{i.morris@surrey.ac.uk}

\maketitle
\begin{abstract}

Furstenberg, Katznelson and Weiss proved in the early 1980s that every measurable subset of the plane with positive density at infinity has the property that all sufficiently large real numbers are realised as the Euclidean distance between points in that set. Their proof used ergodic theory to study translations on a space of Lipschitz functions corresponding to closed subsets of the plane, combined with a measure-theoretical argument. We consider an alternative dynamical approach in which the phase space is given by the set of measurable functions from $\mathbb{R}^d$ to $[0,1]$, which we view as a compact subspace of $L^\infty(\mathbb{R}^d)$ in the weak-* topology. The pointwise ergodic theorem for $\mathbb{R}^d$-actions implies that with respect to any translation-invariant measure on this space, almost every function is asymptotically close to a constant function at large scales. This observation leads to a general sufficient condition for a configuration to occur in every set of positive upper Banach density at all sufficiently large scales, extending a recent theorem of B. Bukh. To illustrate the use of this criterion we apply it to prove a new result concerning three-point configurations in measurable subsets of the plane which form the vertices of a triangle with specified area and side length, yielding a new proof of a result related to work of R. Graham. 

Key words and phrases: Euclidean Ramsey theory, measurable sets, pointwise ergodic theorem. MSC Primary 05D10, 22A99, Secondary 37A15, 37A30.
\end{abstract}

\section{Introduction and main results}

Given a Lebesgue measurable subset $A$ of $\mathbb{R}^d$, let us define the \emph{upper density} of $A$ to be the quantity
\[\overline{d}(A):=\limsup_{t \to \infty} \frac{m(A \cap Q(0,t))}{m(Q(0,t))}\]
where $m$ denotes $d$-dimensional Lebesgue measure and $Q(x,t)$ denotes the closed solid cube in $\mathbb{R}^d$ with side length $t$, centre $x$, and sides oriented parallel to the co-ordinate axes. Define the \emph{upper Banach density} of $A$ to be the quantity
\[d^*(A):=\limsup_{t \to \infty} \sup_{x \in \mathbb{R}^d} \frac{m(A \cap Q(x,t))}{m(Q(x,t))}.\]
This article is motivated by the following result which was proved by H. Furstenberg, Y. Katznelson and B. Weiss \cite{FKW} in response to a conjecture of L. A. Sz\'ekely:
\begin{theorem}[Furstenberg--Katznelson--Weiss]\label{FuKaWe}
Let $A \subseteq \mathbb{R}^2$ be a Lebesgue measurable set such that $\overline{d}(A)>0$. Then for all sufficiently large real numbers $t$ we may find points $x, y \in A$ which are separated by a Euclidean distance of precisely $t$.
\end{theorem}
The proof given originally by Furstenberg, Katznelson and Weiss combines an ergodic-theoretic argument based on translations of closed subsets of the plane with a subsequent measure-theoretical argument. Shorter alternative proofs based on demonstrating the positivity of the integral
\[\int_{\mathbb{R}^2 \times S^1}\chi_A(x)\chi_A(x+ty)d(x,y)\]
were subsequently presented by K. Falconer and J. Marstrand \cite{FM} and J. Bourgain \cite{Bour} using techniques from geometric measure theory and Fourier analysis respectively, and a probabilistic proof was recently given by A. Quas \cite{Quas}; the last two of these results require only the weaker condition $d^*(A)>0$. Bourgain in fact proves the following stronger result: if $V$ is a configuration of $d$ points in $\mathbb{R}^d$ which do not lie in a $\left(d-2\right)$-dimensional subspace, and $A \subseteq \mathbb{R}^d$ is a measurable set with positive upper Banach density, then $A$ contains a set isometric to $tV$ for all sufficiently large real numbers $t$. On the other hand, in the case where $V$ consists of three evenly-spaced colinear points Bourgain exhibited an example of a positive-density measurable set $A \subset \mathbb{R}^3$ which is free from isometric copies of $tV$ at an unbounded set of scales $t$. 

In this article we are interested in giving general conditions under which a configuration or property must be satisfied in all positive-density sets at all sufficiently large scales. B. Bukh \cite[Theorem 8]{Bukh} has previously given a sufficient condition of this kind which we now describe. 
Let us say that a \emph{property} is a function $P$ from the set $\mathfrak{M}(\mathbb{R}^d)$ of all Lebesgue measurable subsets of $\mathbb{R}^d$ to $\{0,1\}$. We consider that $A \subset \mathbb{R}^d$ has property $P$ if $P(A)=1$ and does not have property $P$ if $P(A)=0$. We shall say that $A$ has property $P$ at all large scales if $P(t^{-1}A)=1$ for all sufficiently large real numbers $t$. The following definition paraphrases Bukh \cite{Bukh}:
\begin{definition}\label{ssable}
We shall say that a property $P$ has \emph{supersaturable complement} if there exists a function $I_P\colon \mathfrak{M}(\mathbb{R}^d)\to[0,+\infty]$ which satisfies the following seven axioms:
\begin{enumerate}
\item
There exists $\mathfrak{m}(P)>0$ such that if $\overline{d}(A)>\mathfrak{m}(P)$ then $P(A)=1$. 
\item
If $A\subseteq B$ then $I_P(A)\leq I_P(B)$ and $P(A)\leq P(B)$.
\item
If $I_P(A)>0$ then $A$ has property $P$.
\item
For all $v \in \mathbb{R}^d$ we have $P(A)=P(A+v)$ and $I_P(A)=I_P(A+v)$.
\item
There exists $r>0$ such that if all points of $A_1$ are at least distance $r$ away from all points of $A_2$, then $I_P(A_1\cup A_2) \geq I_P(A_1)+I_P(A_2)$ and $P(A_1 \cup A_2)=\max\{P(A_1),P(A_2)\}$.
\item
There exist $\varepsilon>0$ and a strictly positive function $f \colon (0,+\infty)\to\mathbb{R}$ such that if the set
\[\{x \in \mathbb{R}^d \colon m(Q(x,\delta)\cap A)>(1-\varepsilon)m(Q(x,\delta))\}\]
has property $P$, then $I_P(A)\geq f(\delta)$.
\item
If $A\subset Q(0,R)$ then
\[I_P(A) \geq g_P(\varepsilon)I_P\left(\left\{x \in \mathbb{R}^d \colon m(Q(x,\delta)\cap A)>\varepsilon m(Q(x,\delta))\right\}\right)-h_P(\varepsilon,\delta)R^d\]
where $g_P(\varepsilon)>0$ and $\lim_{\delta \to 0}h_P(\varepsilon,\delta)=0$ for each fixed $\varepsilon$.
\end{enumerate}
\end{definition}
Bukh's work in fact considers necessary conditions for measurable sets to \emph{omit} certain structures as opposed to sufficient conditions for measurable sets to \emph{contain} certain structures, and in respect of this the above description inverts the object considered by Bukh: $P$ is a property with supersaturable complement in the above sense if and only if $1-P$ is a supersaturable property in the original sense defined in \cite{Bukh}. Bukh obtains a number of interesting results concerning supersaturable properties, of which we single out the following:
\begin{theorem}[{Bukh, \cite[Theorem 8]{Bukh}}]
Let $P \colon \mathfrak{M}(\mathbb{R}^d)\to\{0,1\}$ be a  property with supersaturable complement such that $\mathfrak{m}(P)<1$. If $A\subset\mathbb{R}^d$ satisfies $\overline{d}(A)>0$, then $A$ has property $P$ at all large scales.
\end{theorem}
Bukh in fact proves a stronger result concerning the simultaneous satisfaction of a finite number of properties $P_1,\ldots,P_n$ with supersaturable complement at widely-differing scales. Since our interest in this article is in direct extensions of Theorem \ref{FuKaWe} we restrict our attention to the case of a single property occuring at all large scales. In this article we shall use ergodic theory to give a much weaker sufficient condition for a configuration to appear at all sufficiently large scales in every positive-density measurable set. 

In order to state our results we require a few items of notation. We shall use the symbol $\mathbb{S}_d$ to denote the set of all $f \in L^\infty(\mathbb{R}^d)$ such that $0 \leq f \leq 1$ almost everywhere, and we equip this set with the weak-* topology  inherited from $L^\infty(\mathbb{R}^d)\simeq L^1(\mathbb{R}^d)^*$ with respect to which it is compact and metrisable. We use the symbol $\mathbf{1} \in \mathbb{S}_d$ to refer to the almost surely constant function with value $1$.
\begin{definition}
We say that a property $P \colon \mathfrak{M}(\mathbb{R}^d)\to\{0,1\}$ is $\delta$-\emph{mild}, where $\delta \in (0,1]$, if it satisfies the following properties:
\begin{enumerate}
\item
There exists an open set $\mathbb{U}\subset \mathbb{S}_d$ such that if $\chi_A \in \mathbb{U}$ then $P(A)=1$. 
\item
The open set $\mathbb{U}$ contains $\delta\mathbf{1}$.
\item
For all $v \in \mathbb{R}^d$ we have $P(A)=P(A+v)$.
\end{enumerate}
We shall say that a property is \emph{mild} if it is $\delta$-mild for every $\delta>0$. \end{definition}
The main result of this article is the following:
\begin{theorem}\label{milpro}
Let $P \colon \mathfrak{M}(\mathbb{R}^d) \to \{0,1\}$ be a $\delta$-mild property. If the upper Banach density of $A\subseteq\mathbb{R}^d$ is equal to $\delta$ then $A$ has property $P$ at all large scales.
\end{theorem}
Theorem \ref{FuKaWe} may easily be deduced from Theorem \ref{milpro} via the following proposition, which implies that the property of containing two points at Euclidean distance $1$ from one another is a mild property. Proposition \ref{open} is proved in \S\ref{dongle} below. Here and throughout the article we let $\sigma$ denote normalised one-dimensional Lebesgue measure on $S^1\subset\mathbb{R}^2$.
\begin{proposition}\label{open}
The set
\[\mathbb{U} :=\left\{g \in \mathbb{S}_2 \colon \iint g(x)g(x-y)d\sigma(y)dx>0\right\}\]
is open and contains $\delta \mathbf{1}$ for every $\delta \in (0,1]$.
\end{proposition}
It is not difficult to see that $\delta$-mild properties are weaker than properties with supersaturable complement in several respects: for example, the property that $A$ has a translate $A+v$ such that $\frac{1}{4}<m((A+v) \cap [0,1]^d)<\frac{3}{4}$ is clearly $\frac{1}{2}$-mild --- for example, we could define $\mathbb{U}:=\{f \in \mathbb{S}_2 \colon \frac{1}{4}<\int_{[0,1]^2}f<\frac{3}{4}\}$ --- but does not satisfy Definition \ref{ssable}(ii). On the other hand, if $P$ is a property with supersaturable complement such that $\mathfrak{m}(P)<1$ then it is necessarily mild. Given $f \in \mathbb{S}_d$ and a property $P$ with supersaturable complement, let us say that $P(f)=1$ if for every Lebesgue measurable function $g \colon \mathbb{R}^d \to [0,1]$ which is almost everywhere equal to $f$ we have $P\left(\left\{x \colon g(x)>0\right\}\right)=1$. It follows from \cite[Lemma 12]{Bukh} that the set $\mathbb{U}:=\{f \in \mathbb{S}_d \colon P(f)=1\}$ is open, and by Definition \ref{ssable}(i) it follows that $\delta \mathbf{1}\in\mathbb{U}$ for every $\delta \in (0,1]$. Thus if a property has supersaturable complement then it is mild, but the converse implication is not true in general.

 The following application of Theorem \ref{milpro} gives a somewhat less contrived example of a property which is mild but does not have supersaturable complement:
\begin{theorem}\label{grah}
For every $v \in S^1\subset\mathbb{R}^2$ let $v^\perp$ denote the point on $S^1$ which lies anticlockwise from $v$ at a distance of one quarter-circle. For every $M>0$ the set of all $g \in \mathbb{S}_2$ such that the integral
\begin{equation}\label{wang'aa}\int_{\mathbb{R}^2}\int_{S^1} \int_0^\infty e^{-s}g(x)g(x+y)g\left(x+2\alpha y^\perp + s y\right)ds d\sigma(y) dx\end{equation}
is nonzero for every $\alpha \in (0,M]$ contains an open neighbourhood of $\delta\mathbf{1}$ for every $\delta \in (0,1]$. 
In particular, the property $P_M$ defined by $P_M(A)=1$ if and only if for each $\alpha \in (0,M]$ we may find points $x, y, z \in A$ which form the vertices of a triangle with area $\alpha$ and in which at least one side has length exactly $1$ is a mild property.
\end{theorem}
The proof of Theorem \ref{grah} is given in \S\ref{try} below. The reader should not have difficulty in constructing for each sufficiently large $r>0$ a pair of measurable sets $A_1,A_2 \subset \mathbb{R}^2$ such that every point of $A_1$ is separated from every point of $A_2$ by at least distance $r$, and such that $A_1 \cup A_2$ has the property $P_1$ defined in Theorem \ref{grah} but each of $A_1$ and $A_2$ individually does not. This shows that the property described in Theorem \ref{grah} does not have supersaturable complement since it does not satisfy Definition \ref{ssable}(v). 

 In addition to generalising Theorem \ref{FuKaWe} and illustrating the separation between mild properties and properties with supersaturable complement, this result yields the following direct corollary:
\begin{corollary}\label{ggg}
Let $A \subseteq \mathbb{R}^2$ be a Lebesgue measurable set such that $d^*(A)>0$, and let $\alpha$ be any positive real number. Then there exist points $x,y,z \in A$ which form the vertices of a triangle of area $\alpha$. 
\end{corollary}
A proof of Corollary \ref{ggg} based on Szemer\'edi's theorem was previously given by R. L. Graham \cite{G}. Graham's result also has the stronger feature that the vectors $x-z$ and $y-z$ may be chosen parallel to the co-ordinate axes.

\section{Dynamical formulation and technical results}

In the tradition of earlier ergodic-theoretic investigations of translation-invariant combinatorial structures our proof of Theorem \ref{milpro} operates by investigating the translation dynamics on a phase space comprising a compactification of the set of indicator functions of the sets of interest. In Furstenberg, Katznelson and Weiss' investigations of subsets of $\mathbb{Z}^d$ (see for example \cite{Fu77,FuKa78,FuWe78}) the set of all indicator functions $\mathbb{Z}^d \to \{0,1\}$ is already compact in the infinite product topology on $\{0,1\}^{\mathbb{Z}^d}$ and so no enlargement of this space of functions is necessary. Furstenberg, Katznelson and Weiss' investigation of subsets of $\mathbb{R}^d$, on the other hand, substitutes for the indicator function $\mathbb{R}^d \to \{0,1\}$ of a measurable set a Lipschitz continuous function $\mathbb{R}^d \to [0,1]$ given by the distance to the closure of the set in question, and equips the set of such functions with the compact-uniform topology. In this article we take the alternative approach of granting the set of measurable functions $\mathbb{R}^d \to \{0,1\}$ the topology which it inherits as a subset of $L^\infty(\mathbb{R}^d) \simeq L^1(\mathbb{R}^d)^*$ in the weak-* topology. (As Furstenberg, Katznelson and Weiss' approach to subsets of $\mathbb{R}^d$ does not distinguish between sets with different closures, so our approach does not distinguish between sets which agree up to measure zero.) We remark that in the context of functions $\mathbb{Z}^d \to \{0,1\}$ these two approaches would be indistinguishable, since in that environment the infinite product topology, the compact-uniform topology and the weak-* topology inherited from $L^1(\mathbb{Z}^d)^*$ are all coincident.

 For each  $v \in \mathbb{R}^d$ we define a function $\mathcal{T}_v \colon \mathbb{S}_d \to \mathbb{S}_d$ by $(\mathcal{T}_{v}f)(x)=f(x+v)$, and for each $t>0$ we also define a map $\mathcal{Z}_t \colon \mathbb{S}_d \to \mathbb{S}_d$ by $\left(\mathcal{Z}_t f\right)(x)=f(t x)$. It is not difficult to verify that each $\mathcal{T}_v$ and each $\mathcal{Z}_t$ is a homeomorphism, that $v \mapsto \mathcal{T}_v$ is an action of $\mathbb{R}^d$, that $\mathcal{Z}_t\mathcal{T}_v=\mathcal{T}_{tv}\mathcal{Z}_t$ for every $t$ and $v$, and that $\mathcal{T}_{v_n}$ converges uniformly to $\mathcal{T}_v$ in the limit as $v_n \to v$. We denote the collection of maps $\mathcal{T}_v$ simply by $\mathcal{T}$, and say that a set $B \subset \mathbb{S}_d$ is \emph{$\mathcal{T}$-invariant} if $\mathcal{T}_vB=B$ for all $v \in \mathbb{R}^d$. Finally, let us define the \emph{upper Banach density} of a function $f \in \mathbb{S}$ to be the quantity
\[d^*(f):=\limsup_{t \to \infty} \sup_{v \in \mathbb{R}^d} \frac{1}{m(Q(v,t))}\int_{Q(v,t)}f(x)\,dx\]
which is of course analogous to the upper Banach density of a set: if $f=\chi_A$ for some Lebesgue measurable set $A \subseteq \mathbb{R}^d$ then $d^*(A)=d^*(\chi_A)$.  Theorem \ref{milpro} is a corollary of the following:
\begin{theorem}\label{mmaaiinn}
Let $f \in \mathbb{S}_d$ and let $\mathbb{V} \subset \mathbb{S}_d$ be a $\mathcal{T}$-invariant open set which contains the constant function $d^*(f)\mathbf{1}$. Then $\mathcal{Z}_tf \in \mathbb{V}$ for all sufficiently large $t$.
\end{theorem}
To obtain Theorem \ref{milpro} for a given $\delta$-mild property $P$ it suffices to apply Theorem \ref{mmaaiinn} to the set $\mathbb{V}:=\bigcup_{v \in \mathbb{R}^d} \mathcal{T}_v\mathbb{U}$ and function $f:=\chi_A$, noticing that $\mathcal{Z}_t\chi_A=\chi_{t^{-1}A}$ for all $t>0$. We derive Theorem \ref{mmaaiinn} using a dynamical argument in two parts: the first part characterises the density of a set $f$ in terms of space averages over $\mathbb{S}_d$ with respect to translation-invariant measures, and the second shows that with respect to any ergodic translation-invariant measure, almost every element of $\mathbb{S}_d$ is approximately constant at large scales in a precise sense.
 In order to describe these results we require some further definitions. Let us use the symbol $\mathcal{M}$ to denote the set of all Borel probability measures on $\mathbb{S}_d$, which we equip with the weak-* topology arising from that set's identification with a subset of $C(\mathbb{S}_d)^*$ via the Riesz representation theorem. It follows from the Banach-Alaoglu theorem and the separability of $C(\mathbb{S}_d)$ that $\mathcal{M}$ is compact and metrisable in this topology. 
We let $\mathcal{M}_{\mathcal{T}}$ denote the set of all $\mathcal{T}$-invariant Borel probability measures on $\mathbb{S}_d$, which is a nonempty closed subset of $\mathcal{M}$. The following theorem characterises the density $d^*(f)$ in terms of translation-invariant measures:
\begin{theorem}\label{yapapa}
Let $f \in \mathbb{S}_d$ and define $X_f:=\overline{\left\{\mathcal{T}_vf \colon v \in \mathbb{R}^d\right\}}$.
 Then
\begin{equation}\label{beh}d^*(f)=\sup\left\{\iint_{[0,1]^d}h(x)\,d\mu(h)\colon \mu \in \mathcal{M}_{\mathcal{T}}\,\text{ and }\mu\left(X_f\right)=1\right\}\end{equation}
and this supremum is attained by an ergodic measure.
\end{theorem}
The observation that the upper density of a function (or rather, the characteristic function of a set) is positive if and only if an associated ergodic average is positive for at least one ergodic measure is a staple of ergodic Ramsey theory and arises in numerous works on the topic such as \cite{BeLe96,Fu77,Fu81}. The fact that the upper Banach density is exactly characterised by a supremum over ergodic measures in the manner of Theorem \ref{yapapa} seems to be relatively unremarked, though we are aware of \cite[Lemma 1]{Pe88}. 
In the immediate context of subsets of $\mathbb{R}^d$ the above result is not dissimilar to \cite[Lemma 2.1]{FKW}, although that result is prevented from being an equation by the possibility that a measurable set may have strictly lower density than its closure.

The following result relating a space average of an invariant measure $\mu$ to the behaviour of $\mu$-typical elements at large scales is a straightforward consequence of the pointwise ergodic theorem applied to the dynamics of the action $\mathcal{T}$ on the phase space $\mathbb{S}_d$.
\begin{theorem}\label{yamama}
Let $\mu \in \mathcal{M}_{\mathcal{T}}$ be an ergodic measure. Then
\[\mu\left(\left\{f \in \mathbb{S}_d \colon \lim_{t \to \infty}\mathcal{Z}_tf = \left(\iint_{[0,1]^d}h(x)dx \,d\mu(h)\right)\cdot \mathbf{1} \right\}\right)=1.\]
\end{theorem}
Let us briefly indicate the derivation of Theorem \ref{mmaaiinn} from Theorems \ref{yapapa} and \ref{yamama}. Let $\mathbb{U}$, $\delta$ and $f$ be given. By Theorem \ref{yapapa} there exists an ergodic measure $\mu$ such that $\mu(X_f)=1$ and $\iint_{[0,1]^d}h(x)\,d\mu(h)=d^*(f)$. It follows from Theorem \ref{yamama} that there exists $h \in X_f$ such that $\lim_{t \to \infty} \mathcal{Z}_th =d^*(f)\mathbf{1}$. Since $d^*(f)\mathbf{1}$ belongs to the open set $\mathbb{U}$ we have $\mathcal{Z}_th \in \mathbb{U}$ for all sufficiently large $t$. Since $\mathbb{U}$ is open and $h \in X_f$ it follows that for each such $t$ there exists $v \in \mathbb{R}^d$ such that $ \mathcal{Z}_t\mathcal{T}_vf\in \mathbb{U}$, and since $\mathcal{T}_{tv}\mathcal{Z}_tf=\mathcal{Z}_t\mathcal{T}_vf$ and $\mathcal{T}_{tv}^{-1}\mathbb{U}=\mathbb{U}$ it follows that $\mathcal{Z}_tf \in \mathbb{U}$ for all such $t$ as required.

The following two sections comprise the proofs of Theorems \ref{yamama} and \ref{yapapa}, and the two sections subsequent to those contain the proofs of Proposition \ref{open} and Theorem \ref{grah}. Throughout these sections we will frequently use Tonelli and Fubini's theorems without comment.

\section{Proof of Theorem \ref{yamama}}

While it is perhaps more natural to state Theorem \ref{yapapa} before Theorem \ref{yamama}, the proof of the former requires the latter so we shall prove the second theorem first. The following general ergodic theorem due to E. Lindenstrauss \cite{Lin} is convenient for our argument (though see remark below). 
Here and throughout, the expression $A \triangle B$ denotes the symmetric set difference $(A \setminus B) \cup (B \setminus A)$.
\begin{theorem}[Lindenstrauss]\label{dougie}
Suppose that $\Gamma$ is a locally compact, second countable amenable group which acts bi-measurably on the left on a probability  space $(X,\mathcal{B},\mu)$ by measure-preserving transformations, and let $m$ denote Haar measure on $\Gamma$. Suppose that $(F_n)$ is a sequence of compact subsets of $\Gamma$ with the following two properties: firstly, for every nonempty compact set $K \subset \Gamma$
\[\lim_{n \to \infty} \frac{m(F_n \triangle KF_n)}{m(F_n)}=0,\]
and secondly, there exists a constant $C>0$ such that for all $n \geq 2$
\[m\left(F_n^{-1} \bigcup_{i=1}^{n-1}F_i\right) \leq Cm(F_n).\]
Suppose finally that the action of $\Gamma$ is ergodic. Then for every $f \in L^1(\mu)$ 
\[\mu\left(\left\{x \in X \colon \lim_{n \to \infty} \frac{1}{m(F_n)}\int_{F_n}f(gx) dm(g) = \int f\,d\mu\right\}\right)=1.\]
\end{theorem}
We require only the case in which $\Gamma=\mathbb{R}^d$ acts continuously on a compact metrisable space by homeomorphisms, in which case the measurability hypotheses are satisfied trivially.

\begin{proof}[Proof of Theorem \ref{yamama}]
 Let us define a \emph{rational cube} to be a compact set $Q \subseteq \mathbb{R}^d$ which is equal to the Cartesian product of $d$ closed intervals with rational endpoints and equal, nonzero lengths, and denote the set of all rational cubes by $\mathfrak{Q}$. We claim that for every $Q \in \mathfrak{Q}$,
\begin{equation}\label{boo}\mu\left(\left\{f \in \mathbb{S}_d \colon \lim_{n \to \infty}\int \chi_Q \mathcal{Z}_nf = m(Q)\iint_{[0,1]^d}h(x)\,dxd\mu(h)\right\}\right)=1.\end{equation}
Before we prove the claim let us show that the truth of the claim implies the truth of the theorem. Firstly we observe that for any fixed rational cube $Q$ and function $f \in \mathbb{S}_d$ we have for all $t>1$
\begin{align*}\left|\int \chi_Q \mathcal{Z}_t f -  \int \chi_Q \mathcal{Z}_{\lfloor t\rfloor} f \right| &=\left|\frac{1}{t^d}\int_{tQ} f -  \frac{1}{\lfloor t \rfloor^d}\int_{\lfloor t \rfloor Q}  f \right| \\
&\leq \frac{1}{t^d} m(tQ \triangle \lfloor t \rfloor Q) + \left(\frac{1}{\lfloor t \rfloor^d} -\frac{1}{t^d}\right)m(\lfloor t \rfloor Q)\\
&\leq m(Q \triangle t^{-1}\lfloor t \rfloor Q) + dt^{-1}  m(Q).\end{align*}
Since this expression converges to zero as $t \to \infty$ it follows that the truth of the claim implies 
\[\mu\left(\left\{f \in \mathbb{S}_d \colon \lim_{t \to \infty}\int \chi_Q \mathcal{Z}_tf=m(Q)\iint_{[0,1]^d}h(x)\,dxd\mu(h)\right\}\right)=1\]
for every rational cube $Q \in \mathfrak{Q}$. The set $\mathfrak{Q}$ being countable, this in turn implies 
\[\mu\left(\left\{f \in \mathbb{S}_d \colon \lim_{t \to \infty}\int \chi_Q \mathcal{Z}_tf = m(Q)\iint_{[0,1]^d}h(x)\,dxd\mu(h)\text{ for all }Q \in \mathfrak{Q}\right\}\right)=1,\]
and since the linear span of the set of all characteristic functions of rational cubes is dense in $L^1(\mathbb{R}^d)$ it follows by a simple approximation argument that
 \[\mu\left(\left\{f \in \mathbb{S}_d \colon \lim_{t \to \infty}\int \varphi\mathcal{Z}_tf =\iint_{[0,1]^d}h(x)\,dxd\mu(h)\int\varphi\text{ for all }\varphi \in L^1(\mathbb{R}^d)\right\}\right)=1.\]
Since this is by definition equivalent to the statement 
\[\mu\left(\left\{f \in \mathbb{S}_d \colon \lim_{t \to \infty}\mathcal{Z}_tf =\iint_{[0,1]^d}h(x)\,dxd\mu(h)\cdot \mathbf{1} \right\}\right)=1\]
we conclude that the truth of the claim implies the truth of the theorem.

Let us now prove the claim. For the remainder of the proof we fix a rational cube $Q \subset \mathbb{R}^d$ with side length $r$ and centre point $v$. Let us write $|(u_1,\ldots,u_d)|_\infty:=\max\{|u_1|,\ldots,|u_d|\}$ for all $(u_1,\ldots,u_d) \in \mathbb{R}^d$, and define $F_n:=nQ$ for all $n \geq 1$. If $K \subseteq \mathbb{R}^d$ is any compact set then clearly
\begin{align*}
\lim_{n \to \infty} \frac{m(F_n \triangle (K+F_n))}{m(F_n)}&=\lim_{n \to \infty}\frac{m(nQ \triangle (K+nQ))}{n^d m(Q)}\\
&=\lim_{n \to \infty}\frac{m(Q \triangle (n^{-1}K+Q))}{m(Q)}=0.\end{align*}
If $n \geq 2$ and $u \in F_k - F_n$ for $k \in \{1,\ldots,n-1\}$ then clearly $|u|_\infty \leq n(|v|_\infty+r)$, so for all $n \geq 2$ we have
\[m\left(\left(\bigcup_{k=1}^{n-1}F_k\right)-F_n\right) \leq 2^dn^d(|v|_\infty +r)^d =\left(\frac{(2|v|_\infty+2r)^d}{m(Q)}\right)m(F_n).\]
The sequence $(F_n)$ therefore satisfies the requirements of Theorem \ref{dougie} with respect to the group $\Gamma=\mathbb{R}^d$. Let $\Phi \colon \mathbb{S}_d \to \mathbb{R}$ be the functional $\Phi(f):=\int_{[0,1]^d}f$ which is continuous by the definition of the topology on $\mathbb{S}_d$, and let $\mathcal{B}$ denote the Borel $\sigma$-algebra on $\mathbb{S}_d$. Applying Theorem \ref{dougie} to the measure space $(\mathbb{S}_d,\mathcal{B},\mu)$, group $\Gamma=\mathbb{R}^d$, action $v \mapsto \mathcal{T}_v$ and function $\Phi$ we obtain
\begin{equation}\label{oob}\mu\left(\left\{f \in \mathbb{S}_d \colon \lim_{n \to \infty} \frac{1}{m(nQ)}\int_{nQ} \Phi(\mathcal{T}_vf)dv = \iint_{[0,1]^d}h(x)\,dxd\mu(h)\right\}\right)=1.\end{equation}
It remains only to show that this is equivalent to the claimed expression \eqref{boo}. Let us define
\[\mathcal{Q}_n^+:=\bigcup_{x \in [0,1]^d} nQ-x,\qquad\mathcal{Q}_n^-:=\bigcap_{x \in [0,1]^d} nQ-x\]
for every $n \geq 1$. When $nr-1$ is non-negative the sets $\mathcal{Q}_n^{+}$ and $\mathcal{Q}_n^-$ are rational cubes with side length respectively $nr+1$ and $nr-1$ such that $\mathcal{Q}_n^- \subset nQ \subset \mathcal{Q}_n^+$. In particular we have $m(\mathcal{Q}_n^+ \setminus \mathcal{Q}_n^-)=(nr+1)^d-(nr)^d < 2^d(nr)^{d-1}$ for all $n \geq 1/r$.
For each $n \geq 1$ we have
\[\int_{nQ} \Phi(\mathcal{T}_vf)dv = \int_{nQ} \int_{[0,1]^d}f(x+v)dxdv = \int_{[0,1]^d}\int_{nQ-x} f(v)dvdx\]
and therefore in particular
\[\frac{1}{m(nQ)}\int_{\mathcal{Q}_n^-}f(v)dv \leq \frac{1}{m(nQ)}\int_{nQ} \Phi(\mathcal{T}_vf) dv \leq \frac{1}{m(nQ)}\int_{\mathcal{Q}_n^+}f(v)dv \]
and
\[\frac{1}{m(nQ)}\int_{\mathcal{Q}_n^-}f(v)dv \leq \frac{1}{m(nQ)}\int \chi_{nQ} f =\frac{1}{m(Q)}\int \chi_Q \mathcal{Z}_nf \leq \frac{1}{m(nQ)}\int_{\mathcal{Q}_n^+}f(v)dv\]
for all $n \geq 1/r$. Since
\[0 \leq \limsup_{n \to \infty} \frac{1}{m(nQ)} \left(\int_{\mathcal{Q}_n^+}f(v)dv- \int_{\mathcal{Q}_n^-}f(v)dv\right)\leq \lim_{n \to \infty} \frac{m(\mathcal{Q}_n^+\setminus \mathcal{Q}_n^-)}{n^dm(Q)}=0\]
we deduce that
\[\lim_{n \to \infty}\left|\frac{1}{m(nQ)}\int_{nQ} \Phi(\mathcal{T}_vf) dv - \frac{1}{m(Q)}\int \chi_{Q} \mathcal{Z}_nf\right|=0\]
for every $f \in \mathbb{S}_d$. We conclude that \eqref{oob} implies \eqref{boo} and the theorem is proved.
\end{proof}


\emph{Remark}. As an alternative to Lindenstrauss' ergodic theorem it would also have been possible for us to use a much earlier theorem due to A. Tempelman (\cite{Tempel67}, see also \cite{Bewley,Temp92}). Tempelman's result however cannot be applied directly to the sequence $(nQ)$ since it requires the additional hypotheses that the sequence $(F_n)$ must be nested and each $F_n$ must include the origin. To apply Tempelman's theorem in our context we would have to rewrite $\chi_{nQ}$ as a linear combination up to measure zero of $2^d$ characteristic functions of closed cuboids $nK_i$ each containing the origin in its boundary, apply Tempelman's theorem individually to all of the sequences $(nK_i)$, and then sum up to obtain the result \eqref{oob}. For example, to treat the case $d=1$ in this manner we would study the expression $\int_{\alpha n}^{\beta n } \Phi(\mathcal{T}_vf)dv$ by writing it as the difference $\int_0^{\beta n } \Phi(\mathcal{T}_vf)dv - \int_0^{\alpha n} \Phi(\mathcal{T}_vf)dv$.


\section{Proof of Theorem \ref{yapapa}}
Let us use the notation $\mathcal{M}_{\mathcal{T}}(X_f)$ to denote the set of all $\mathcal{T}$-invariant Borel probability measures on $\mathbb{S}_d$ which give full measure to the compact nonempty set $X_f$. By the Krylov-Bogolioubov Theorem $\mathcal{M}_{\mathcal{T}}(X_f)$ is nonempty. We claim that given any $\mu \in \mathcal{M}_{\mathcal{T}}(X_f)$ we may find an ergodic measure $\hat\mu \in \mathcal{M}_{\mathcal{T}}(X_f)$ such that \begin{equation}\label{foo}\iint_{[0,1]^d}h(x)\,dxd\hat\mu(h)\geq\iint_{[0,1]^d}h(x)\,dxd\mu(h).\end{equation}

By the definition of the topology of $\mathbb{S}_d$ the function $h \mapsto \int_{[0,1]^d}h$ is a continuous map from $\mathbb{S}_d$ to $\mathbb{R}$. The function $\nu \mapsto \iint_{[0,1]^d}h(x)dx\,d\nu(h)$ is thus a continuous affine functional on the metrisable topological space $\mathcal{M}_{\mathcal{T}}(X_f)$, which is a compact convex subspace of the locally convex topological space $C(\mathbb{S}_d)^*$ equipped with its weak-* topology. By applying a suitable version of Choquet's theorem to the measure $\mu$ (see e.g. \cite[p.14]{Phelps}) it follows that there exists a Borel probability measure $\mathbb{P}$ on $\mathcal{M}_{\mathcal{T}}(X_f)$ such that
\[\iiint_{[0,1]^d}h(x)\,dxd\nu(h)d\mathbb{P}(\nu)=\iint_{[0,1]^d}h(x)\,dxd\mu(h)\]
and $\mathbb{P}$ gives full measure to the set of \emph{extreme points} of $\mathcal{M}_{\mathcal{T}}(X_f)$, which is defined to be the set of all elements of $\mathcal{M}_{\mathcal{T}}(X_f)$ which may not be written as a strict linear combination of two distinct elements of $\mathcal{M}_{\mathcal{T}}(X_f)$. Consequently $\mathbb{P}$ gives nonzero measure to the set of all $\hat\mu \in \mathcal{M}_{\mathcal{T}}(X_f)$ which are extreme points and also satisfy \eqref{foo}, and in particular this set is nonempty. Choose any measure $\hat\mu$ belonging to this set. If $\hat\mu$ is not ergodic, then there exists a measurable set $A \subset \mathbb{S}_d$ such that $0<\hat\mu(A)<1$ and $\mathcal{T}_vA=A$ up to measure zero for all $v \in \mathbb{R}^d$. In this instance we may then write $\hat\mu$ as a strict linear combination of the distinct measures $\nu_1,\nu_2 \in \mathcal{M}_{\mathcal{T}}(X_f)$ defined by $\nu_1(B):=\hat\mu(A)^{-1}\hat\mu(B \cap A)$, $\nu_2(B):=\hat\mu(A)^{-1}\hat\mu(B \setminus A)$ for Borel sets $B \subseteq \mathbb{S}_d$. By virtue of its being an extreme point $\hat\mu$ cannot be written as strict linear combination of distinct invariant measures and it follows that $\hat\mu$ is ergodic. This completes the proof of the claim.

Let us now prove the theorem. Using the weak-* continuity of the functional $g \mapsto \int_{[0,t]^d}g$ we obtain
\[d^*(f)=\limsup_{t \to\infty} \sup_{v \in \mathbb{R}^d} \frac{1}{m([0,t]^d)}\int_{[0,t]^d}\mathcal{T}_vf = \limsup_{t \to \infty} \sup_{g \in X_f} \frac{1}{t^d} \int_{[0,t]^d} g.\]
Given $\mu \in \mathcal{M}_{\mathcal{T}}(X_f)$, choose an ergodic measure $\hat\mu\in\mathcal{M}_{\mathcal{T}}(X_f)$ such that \eqref{foo} holds. Using Theorem \ref{yamama} we find that for $\hat\mu$-almost-every $g$ we have $g \in X_f$ and
\[\lim_{t \to \infty} \frac{1}{t^d} \int_{[0,t]^d} g=\lim_{t \to \infty} \int_{[0,1]^d} \mathcal{Z}_t g =\iint_{[0,1]^d}h(x)\,dxd\hat\mu(h)\geq \iint_{[0,1]^d}h(x)\,dxd\mu(h).\]
Since $\mu$ is arbitrary it follows by combining these two expressions that $d^*(f)$ is greater than or equal to the right-hand member of \eqref{beh}.

Let us now prove the opposite inequality. Let $(t_n)_{n=1}^\infty$ be a sequence of real numbers tending to infinity and $(v_n)_{n=1}^\infty$ a sequence of vectors in $\mathbb{R}^d$ such that $d^*(f)=\lim_{n \to \infty} t_n^{-d} \int_{[0,t_n]^d}\mathcal{T}_{v_n}f$. Without loss of generality we assume that $t_n>1$ for every $n$. For each $n \geq 1$ define a Borel probability measure on $X_f$ by
\[\mu_n:=\frac{1}{t_n^d}\int_{[0,t_n]^d} \delta_{\mathcal{T}_{v_n+u}f} du.\]
Since $\mathcal{M}$ is compact and metrisable we may choose a strictly increasing sequence of integers $(n_j)_{j=1}^\infty$ and measure $\mu \in \mathcal{M}$ such that $\lim_{j \to \infty}\mu_{n_j}=\mu$. For each $n \geq 1$ we have
\begin{align*}\iint_{[0,1]^d}h(x)\,d\mu_n(h) &=\int_{[0,t_n]^d}\left(\int_{[0,1]^d}\mathcal{T}_{v_n+u}f(x)dx\right)du\\
&=\int_{[0,1]^d}\int_{[0,t_n]^d}\mathcal{T}_{v_n}f(x+u)du dx\\
&\geq \int_{[0,t_n-1]^d}\mathcal{T}_{v_n}f(w)dw\end{align*}
where we have used the fact that $ [0,t_n-1]^d \subset [0,t_n]^d-x $ for all $x \in [0,1]^d$. Since we additionally have for each $n \geq 1$
\[ \frac{1}{t_n^d}\left|\int_{[0,t_n-1]^d}\mathcal{T}_{v_n}f(u)du- \int_{[0,t_n]^d}\mathcal{T}_{v_n}f(u)du\right| \leq \frac{m([0,t_n]^d\setminus [0,t_n-1]^d)}{t_n^d}< \frac{2^d}{t_n}\]
it follows that
\begin{align*}\iint_{[0,1]^d}h(x)d\mu(h) &= \lim_{j \to \infty} \iint_{[0,1]^d}h(x)d\mu_{n_j}(h) \\
&\geq \liminf_{n \to \infty} \frac{1}{t_n^d}\int_{[0,t_n-1]^d}\mathcal{T}_{v_n}f(u)du = d^*(f).\end{align*}
It is clear that $\mu_n(X_f)=1$ for every $n$, and since $X_f$ is closed this implies $\mu(X_f)=1$.  We claim that $\mu$ is $\mathcal{T}$-invariant. If $w \in \mathbb{R}^d$ and $\Psi \colon \mathbb{S}_d \to \mathbb{R}$ is continuous, then
\begin{eqnarray*}\lefteqn{\left|\int \Psi(h)\,d\mu(h)- \int \Psi(\mathcal{T}_w h)\,d\mu(h)\right|}\\
&&=\lim_{j \to \infty} \left|\int \Psi(h)\,d\mu_{n_j}(h)- \int \Psi(\mathcal{T}_w h)\,d\mu_{n_j}(h)\right|\\
&&\leq \lim_{j \to \infty} \frac{1}{t_{n_j}^d}\left|\int_{[0,t_n]^d} \Psi\left(\mathcal{T}_{v_n+u}f\right) du -\int_{[0,t_n]^d}\Psi\left(\mathcal{T}_{v_n+u+w}f\right)du\right|\\
&&\leq \lim_{n \to \infty} \frac{|\Psi|_\infty}{t_{n}^d} m\left([0,t_n]^d \triangle \left([0,t_n]^d -w\right)\right)\\
&&= \lim_{n \to \infty} |\Psi|_\infty m\left([0,1]^d \triangle \left([0,1]^d -t_n^{-1}w\right)\right)=0.\end{eqnarray*}
Since $\Psi$ is arbitrary it follows that $\mu=\mathcal{T}_w^*\mu$, and since $w$ is arbitrary we conclude that $\mu$ is $\mathcal{T}$-invariant and hence attains the supremum in \eqref{beh}. Replacing $\mu$ with an ergodic measure $\hat\mu\in\mathcal{M}_{\mathcal{T}}(X_f)$ which satisfies \eqref{foo} completes the proof.

\section{Proof of Proposition \ref{open}}\label{dongle}
Clearly the set $\mathbb{U}$ contains $\delta \mathbf{1}$ for all $\delta \in (0,1]$. To prove the set's openness we follow a Fourier-analytic approach suggested by \cite{Bour,Bukh}. Let us define $I(g):=\iint g(x)g(x-y)d\sigma(y)dx$ for all $g \in \mathbb{S}_2$, and for each set $B\subset \mathbb{R}^2$ with finite measure define $I_B(g):=I(\chi_B g)$. Using the monotone convergence theorem one may easily show that $I(g)$ equals the supremum of $I_B(g)$ over all finite-measure  sets $B$, so to show that $\mathbb{U}$ is open it is sufficient to show that each function $I_B \colon \mathbb{S}_2 \to \mathbb{R}$ is continuous. Let $B$ be a bounded measurable set and suppose that $(f_n)$ is a sequence in $\mathbb{S}_2$ converging to $f$. We may write
\begin{align*}I_B(f_n)=\int (\chi_B f_n)(x) ((\chi_B f_n) * \sigma)(x) dx &= \int \widehat{\chi_B f_n}(\xi) \overline{\widehat{\chi_B f_n}(\xi)\hat{\sigma}(\xi)}d\xi\\&=\int |\widehat{\chi_B f_n}(\xi)|^2 \hat{\sigma}(\|\xi\|)d\xi\end{align*}
using Parseval's theorem, with a similar identity holding for $\chi_Bf$. Since $f_n$ converges in the weak-* topology to $f$ it follows that $\widehat{\chi_B f_n}$ converges pointwise to $\widehat{\chi_B f}$. The sequence of integrands $|\widehat{\chi_B f_n}(\xi)|^2 \hat{\sigma}(\|\xi\|)$ is uniformly bounded since $|\widehat{\chi_B f_n}(\xi)| \leq m(B)$ and $| \hat{\sigma}(\|\xi\|)|\leq 1$ for all $\xi \in \mathbb{R}^2$ and $n \geq 1$. Since $\hat{\sigma}(\|\xi\|)$ tends to zero as $\|\xi\| \to \infty$, we may for each $\varepsilon>0$ find a bounded measurable set $K\subset \mathbb{R}^2$ such that $|\hat{\sigma}(\xi)|\leq \varepsilon $ for all $\xi \in \mathbb{R}^2 \setminus K$, and consequently
\[\left|\int_{\mathbb{R}^2 \setminus K}  |\widehat{\chi_B f_n}(\xi)|^2 \hat{\sigma}(\|\xi\|)d\xi\right| \leq \varepsilon \int_{\mathbb{R}^2} |\widehat{\chi_B f_n}(\xi)|^2 d\xi=\varepsilon \|\chi_B f_n\|_2^2 \leq \varepsilon m(B)\]
for all $n \geq 1$ and similarly for $\chi_Bf$. Applying the dominated convergence theorem to the integral over $K$ we deduce that $\limsup_{n \to \infty}| I_B(f_n)-I_B(f)|\leq 2\varepsilon m(B)$, and since $\varepsilon>0$ is arbitrary the result follows.

\section{Proof of Theorem \ref{grah}}\label{try}
We adapt the method of Bourgain \cite[Proposition 3]{Bour}. 
 For notational convenience, for each $y \in S^1$ and $\alpha \in (0,M]$ we define a Borel probability measure $\nu_y^\alpha$ on $\mathbb{R}^2$ by $\nu_y^\alpha(A):=\int_0^\infty e^{-s}\chi_A(2\alpha y^\perp+s y)ds$ for every Borel set $A \subseteq \mathbb{R}^2$. The integral \eqref{wang'aa} may thus be rewritten as
\[\mathbf{D}_1^\alpha(g):=\iiint g(x)g(x+y)g(x+z)d\nu_y^\alpha(z) d\sigma(y)dx.\]
Let us fix $\delta \in (0,1]$ and $M>0$ for the remainder of the proof, and show that $\delta\mathbf{1}$ belongs to the interior of the set
\[\mathbb{U}_M:=\{g \in \mathbb{S}_2 \colon \mathbf{D}_1^\alpha(g)>0 \text{ for all }\alpha \in (0,M]\}.\]
Without loss of generality we assume $M>1$.  To prove the theorem we will show that if $(f_n)$ is a sequence of elements of $\mathbb{S}_2$ which converges to $\delta\mathbf{1}$, then for all sufficiently large $n$ we have $\mathbf{D}_1^\alpha(f_n)>0$ for every $\alpha \in (0,M]$.
To achieve this we will prove the following: there exists a large ball $B \subset \mathbb{R}^2$ such that if $(f_n)$ is a sequence of elements of $\mathbb{S}_2$ which converges to $\delta\chi_B$, and each $f_n$ is supported in $B$, then for all sufficiently large $n$ we have $\mathbf{D}^\alpha_1(f_n)>0$ for all $\alpha \in (0,M]$. To see that this implies the previous statement, note that if $(f_n)$ is a general sequence converging to $\delta\mathbf{1}$ then $(\chi_B.f_n)$ is a sequence supported in $B$ which converges to $\delta\chi_B$, and clearly $\mathbf{D}^\alpha_1(f_n) \geq \mathbf{D}^\alpha_1(\chi_B . f_n)$ for every $n \geq 1$ and $\alpha \in (0,M]$.

We next establish some more notation and fix some parameters. For each $\lambda>0$ let $P_\lambda \colon \mathbb{R}^2 \to \mathbb{R}$ be the unique continuous function whose Fourier transform satisfies $\hat{P}_\lambda(\xi)=e^{-\lambda\|\xi\|}$ for all $\xi \in \mathbb{R}^2$. Note that $P_{\lambda}(rx)=r^{-1}P_{\lambda/r}(x)$ for all $r>0$ and $x \in \mathbb{R}^2$, and by a standard argument we have $\lim_{\lambda \to 0} \|g-(g*P_\lambda)\|_1=0$ for all $g \in L^1(\mathbb{R}^d)$ (see e.g. \cite{Katznelson}). Fix real numbers $\lambda_1,\lambda_2>0$ such that $2\lambda_1^{1/3}<\delta^3/7$ and $12M\lambda_2^{-1}<\delta^3/7$, and for each $r>0$ let $B_r(0)$ denote the closed ball in $\mathbb{R}^2$ centred at the origin with radius $r$. Since
\begin{eqnarray*}\lefteqn{\left\|(\chi_{B_r(0)} * P_{\lambda_1})-(\chi_{B_r(0)} *P_{\lambda_2})\right\|_1} \\& &= \int \left|\int\chi_{B_r(0)}(y)(P_{\lambda_1}-P_{\lambda_2})(x-y)dy\right| dx\\
&&=r^4\int \left|\int\chi_{B_r(0)}(ry)(P_{\lambda_1}-P_{\lambda_2})(rx-ry)dy\right| dx\\
&&=r^2\int \left|\int\chi_{B_1(0)}(y)(P_{\lambda_1/r}-P_{\lambda_2/r})(x-y)dy\right| dx\\
&&=r^2\left\|(\chi_{B_1(0)} * P_{\lambda_1/r})-(\chi_{B_1(0)} * P_{\lambda_2/r})\right\|_1\end{eqnarray*}
for all $r>0$, it follows that
\[\lim_{r \to \infty} m(B_r(0))^{-1}\|(\chi_{B_r(0)} * P_{\lambda_1})- (\chi_{B_r(0)} * P_{\lambda_2})\|_1=0.\]
It is clear that
\[\lim_{r \to \infty}m(B_r(0))^{-1}\mathbf{D}_1^M(\delta \chi_{B_r(0)})=\delta^3,\]
and so by taking $B:=B_r(0)$ for some sufficiently large $r$ we may obtain
\begin{equation}\label{whim}\|(\delta\chi_{B} * P_{\lambda_1})- (\delta\chi_{B} * P_{\lambda_2})\|_1<\frac{1}{7}\delta^3m(B)
\end{equation} and 
\begin{equation}\label{lard}\mathbf{D}_1^M(\delta\chi_B)>\frac{6}{7}\delta^3m(B).\end{equation}
We fix a bounded measurable set $B \subset \mathbb{R}^2$ with these properties for the remainder of the proof.

Having fixed these parameters we may now commence the proof proper. Define $f:=\delta \chi_B$ and let $f_n$ be a sequence of elements of $\mathbb{S}_2$ which converges to $f$ with the property that every $f_n$ is supported in $B$. For each $\alpha>0$ and $g \in \mathbb{S}_2$ supported in $B$ let
\[\mathbf{D}_2^\alpha(g):=\iiint g(x)g(x+y)( g * P_{\lambda_1})(x+z)d\nu_y^\alpha(z) d\sigma(y)dx,\]
\[\mathbf{D}_3^\alpha(g):=\iiint g(x)g(x+y)( g * P_{\lambda_2})(x+z)d\nu_y^\alpha(z) d\sigma(y)dx,\]
\[\mathbf{D}_4(g):=\iint g(x)g(x+y)( g * P_{\lambda_2})(x) d\sigma(y)dx.\]
The Fourier transform $\hat{\nu}_y^\alpha$ of $\nu_y^\alpha$ satisfies
\[\hat{\nu}_y^\alpha(\xi):=\int_0^\infty e^{-2\pi i\langle 2\alpha y^\perp + s y,\xi\rangle -s}ds = \frac{e^{-4\alpha\pi i \langle y^\perp,\xi\rangle}}{1+2\pi i \langle y,\xi\rangle}\]
for all $\xi \in \mathbb{R}^2$, and therefore in particular
\begin{align}\label{boom}\int\left|\hat\nu_y^\alpha(\xi)\right|d\sigma(y) &=\int_0^1 \frac{d\theta}{\sqrt{1+4\pi^2\|\xi\|^2\cos^2(2\pi \theta)}}\\
\notag&\leq\left(\int_0^1 \frac{d\theta}{1+4\pi^2\|\xi\|^2\cos^2(2\pi \theta)}\right)^{1/2} \\
\notag&=\frac{1}{\sqrt[4]{1+4\pi^2\|\xi\|^2}}\\
\notag&\leq \min\left\{1,\frac{1}{\sqrt{\|\xi\|}}\right\}\end{align}
when $\xi$ is nonzero. For each $\alpha \in (0,M]$ we have
\begin{align*}\int \left|\left(g-\left(g * P_{\lambda_1}\right)\right)(x)\right|^2dx &= \int \left|\hat{g}(\xi)\right|^2\left(1-e^{-\lambda_1\|\xi\|}\right)^2d\xi\\ &<\int \left|\hat{g}(\xi)\right|^2d\xi = \|g\|_2^2\leq m(B)\end{align*}
and thus
\begin{align}\label{froggate}\left|\mathbf{D}_1^\alpha(g)-\mathbf{D}_2^\alpha(g)\right|&\leq  \iint\left| ((g - (g * P_{\lambda_1}))*\nu^\alpha_y)(x)\right|dx d\sigma(y)\\\nonumber
&\leq  m(B)^{\frac{1}{2}}\left( \iint\left| ((g - (g * P_{\lambda_1}))*\nu_y^\alpha)(x)\right|^2 dx d\sigma(y)\right)^{\frac{1}{2}}\\\nonumber
&=  m(B)^{\frac{1}{2}}\left( \iint\left|\hat{g}(\xi)\right|^2 \left(1-e^{-\lambda_1\|\xi\|}\right)^2 \left|\hat{\nu}^\alpha_y(\xi)\right| d\xi d\sigma(y)\right)^{\frac{1}{2}}\\\nonumber
&< m(B)^{\frac{1}{2}}\left( \int_{\|\xi\|\leq\lambda_1^{-\frac{4}{3}}}\left|\hat{g}(\xi)\right|^2\lambda_1^2 \|\xi\|^2 d\xi  + \int_{\|\xi\|>\lambda_1^{-\frac{4}{3}}}\frac{\left|\hat{g}(\xi)\right|^2}{\|\xi\|^{\frac{1}{2}}}d\xi\right)^{\frac{1}{2}}\\\nonumber
&< 2 m(B)^{1/2}\lambda_1^{1/3}\|g\|_2 \leq 2\lambda_1^{1/3}m(B)<\frac{1}{7}\delta^3m(B)\end{align}
using the Cauchy-Schwarz inequality, Parseval's theorem and \eqref{boom} together with the trivial bound $|\hat{\nu}_y^\alpha(\xi)| \leq 1$ and our choice of $\lambda_1$. We also clearly have
\begin{equation}\label{horatio}\left|\mathbf{D}^\alpha_2(g)-\mathbf{D}^\alpha_3(g)\right|\leq \|(g * P_{\lambda_1}) - (g * P_{\lambda_2})\|_1.\end{equation}
Since for all $z \in \mathbb{R}^2$
\[\int \left|\left(g * P_{\lambda_2}\right)(x+z)-\left(g * P_{\lambda_2}\right)(x)\right|^2dx \leq 4\|g\|_2^2\leq 4m(B)\]
we furthermore have
\begin{eqnarray}\label{dingus}\lefteqn{\left|\mathbf{D}^\alpha_3(g)-\mathbf{D}_4(g)\right|}\\\nonumber
&&\leq \sup_{y \in S^1} \iint\left| (g * P_{\lambda_2})(x+z)-(g * P_{\lambda_2})(x)\right|dx d\nu_y^\alpha(z)\\\nonumber
&&\leq \sup_{y \in S^1} 2 m(B)^{\frac{1}{2}}\left( \iint\left|  (g * P_{\lambda_2})(x+z)-(g * P_{\lambda_2})(x)\right|^2dx  d\nu_y^\alpha(z) \right)^{\frac{1}{2}}\\\nonumber
&&=\sup_{y \in S^1}2m(B)^{\frac{1}{2}}\left( \iint\left|\hat{g}(\xi)\right|^2\left|1-e^{2\pi i \langle z,\xi\rangle}\right|^2  e^{-2\lambda_2\|\xi\|}  d\xi d\nu_y(z)^\alpha\right)^{\frac{1}{2}}\\\nonumber
&&\leq 4\pi m(B)^{\frac{1}{2}}\sup_{y \in S^1}\left( \int\left|\hat{g}(\xi)\right|^2  \left(\int |\langle z,\xi\rangle|^2 d\nu_y^\alpha(z)\right)e^{-2\lambda_2\|\xi\|}  d\xi\right)^{\frac{1}{2}}\\\nonumber
&&\leq 4\pi m(B)^{\frac{1}{2}}\left( \int\left|\hat{g}(\xi)\right|^2\|\xi\|^2  \left(\int_0^\infty e^{-s}(4\alpha^2+s^2)ds\right)e^{-2\lambda_2\|\xi\|}  d\xi\right)^{\frac{1}{2}}\\\nonumber
&&\leq \left(\frac{4\pi M\sqrt{6}}{e\lambda_2}\right) m(B)^{\frac{1}{2}}\|g\|_2 <12M\lambda_2^{-1}m(B)<\frac{1}{7}\delta^3m(B)\end{eqnarray}
using in turn the Cauchy-Schwarz inequality, Parseval's theorem, Lipschitz continuity, Cauchy-Schwarz again, the elementary inequality $t^2e^{-\lambda_2t}\leq (e\lambda_2)^{-2}$ and our choice of $\lambda_2$. Combining \eqref{froggate}, \eqref{horatio} and \eqref{dingus} we conclude that
\begin{equation}\label{boobs}\left|\mathbf{D}_1^\alpha(g)-\mathbf{D}_4(g)\right|\leq \frac{\delta^3 m(B)}{7}+ \|(g * P_{\lambda_1}) - (g * P_{\lambda_2})\|_1\end{equation}
for every $\alpha \in (0,M]$ and every $g \in \mathbb{S}_2$ which is supported in $B$. Combining \eqref{boobs}, \eqref{whim} and \eqref{lard} we obtain
\begin{equation}\label{sparrow}\mathbf{D}_4(f) >  \mathbf{D}_1^M(f) - \frac{2}{7}\delta^3m(B) - \|(\delta\chi_{B} * P_{\lambda_1})- (\delta\chi_{B} * P_{\lambda_2})\|_1>\frac{3}{7}\delta^3m(B).\end{equation}
Since $f_n$ converges to $f$ in the weak-* topology it follows that $f_n * P_{\lambda_1}$ and $f_n * P_{\lambda_2}$ converge pointwise almost everywhere to $f * P_{\lambda_1}$ and $f * P_{\lambda_2}$ respectively. Since furthermore
\[|(f_n * P_{\lambda_1})(x)- (f_n * P_{\lambda_2})(x)|\leq (\chi_B * P_{\lambda_1})(x)+ (\chi_B * P_{\lambda_2})(x)\]
pointwise almost everywhere, it follows from the dominated convergence theorem together with \eqref{whim} that
\[\|(f_n * P_{\lambda_1})- (f_n * P_{\lambda_2})\|_1 < \frac{1}{7}\delta^3m(B)\]
for all sufficiently large $n$. Applying \eqref{boobs} once more and noticing that this inequality does not depend on $\alpha \in (0,M]$, we find that
\[\inf_{\alpha \in (0,M]}\mathbf{D}_1^\alpha(f_n) >  \mathbf{D}_4(f_n) - \frac{2}{7}\delta^3m(B)\]
when $n$ is sufficiently large. In view of \eqref{sparrow}, it follows that to complete the proof of the theorem it is sufficient to prove that $\lim_{n \to \infty}\mathbf{D}_4(f_n)=\mathbf{D}_4(f)$.

We now prove that this is indeed the case. By Parseval's theorem we have
\[\mathbf{D}_4(f_n)=\int \hat{f_n}(\xi)e^{-\lambda_2\|\xi\|} \left(\widehat{f_n (f_n * \sigma)}\right)(\xi)d\xi\]
for all $n \geq 1$, and a similar expression holds for $f$. For all $\xi \in \mathbb{R}^2$ we have
\[\left|\hat{f_n}(\xi)e^{-\lambda_2\|\xi\|} \left(\widehat{f_n (f_n * \sigma)}\right)(\xi)\right| \leq \|f_n\|_1 e^{-\lambda_2\|\xi\|} \left\|f_n (f_n * \sigma)\right\|_1\leq e^{-\lambda_2\|\xi\|}m(B)^3,\]
and so if we can establish
\begin{equation}\label{bollll}\lim_{n \to \infty}  \hat{f_n}(\xi)e^{-\lambda_2\|\xi\|} \left(\widehat{f_n (f_n * \sigma)}\right)(\xi)= \hat{f}(\xi)e^{-\lambda_2\|\xi\|} \left(\widehat{f (f * \sigma)}\right)(\xi)\end{equation}
for every $\xi \in \mathbb{R}^2$ then it will follow by the dominated convergence theorem that $\mathbf{D}_4(f)=\lim_{n \to \infty}\mathbf{D}_4(f_n)$ as required.
To complete the proof of the theorem we fix $\xi \in \mathbb{R}^2$ and establish \eqref{bollll}. Since each $f_n$ is supported in $B$ it follows from the weak-* convergence of $f_n$ to $f$ that $\lim_{n \to \infty}\hat{f}_n(\eta) = \hat{f}(\eta)$ for all $\eta \in \mathbb{R}^2$. For all $n \geq 1$ we have
\[\widehat{f_n(f_n * \sigma)}(\xi)=\int \hat{f}_n(\eta)\hat{f}_n(\eta-\xi) \hat\sigma(\eta-\xi)d\eta\]
using standard properties of convolutions, with a similar identity for $f$ in place of $f_n$. The Fourier transform $\hat\sigma(\eta)$ is given by the Bessel function $J_0(2\pi\|\eta\|)=\int_0^{2\pi}\cos (2\pi\|\eta\| \sin \theta) d\theta$ which converges to zero as $\|\eta\| \to \infty$. Given any $\varepsilon>0$ we may therefore choose a bounded measurable set $K \subset \mathbb{R}^2$ such that $|\hat\sigma(\eta-\xi)|\leq \varepsilon$ for all $\eta \in \mathbb{R}^2 \setminus K$. Since trivially $|\hat\sigma(\eta)| \leq 1$ for every $\eta$ we obtain for all $n \geq 1$
\begin{align*}\left|\widehat{f_n(f_n * \sigma)}(\xi)-\widehat{f(f * \sigma)}(\xi)\right|\leq& \int_K \left|\hat{f}_n(\eta)\hat{f}_n(\eta-\xi)-\hat{f}(\eta)\hat{f}(\eta-\xi)\right|d\eta \\
&+ \varepsilon\int_{\mathbb{R}^2 \setminus K} \left|\hat{f}_n(\eta)\hat{f}_n(\eta-\xi)-\hat{f}(\eta)\hat{f}(\eta-\xi)\right|d\eta.\end{align*}
Since $\left|\hat{f}_n(\eta)\hat{f}_n(\eta-\xi)-\hat{f}(\eta)\hat{f}(\eta-\xi)\right| \leq \|f_n\|_1^2+\|f\|_1^2 \leq 2m(B)^2$ for all $\eta \in \mathbb{R}^2$ and $n \geq 1$ the dominated convergence theorem yields
\[\lim_{n \to \infty}\int_K \left|\hat{f}_n(\eta)\hat{f}_n(\eta-\xi)-\hat{f}(\eta)\hat{f}(\eta-\xi)\right|d\eta =0.\]
On the other hand, the Cauchy-Schwarz inequality together with Parseval's identity yields
\[\varepsilon\int_{\mathbb{R}^2 \setminus K} \left|\hat{f}_n(\eta)\hat{f}_n(\eta-\xi) - \hat{f}(\eta)\hat{f}(\eta-\xi)\right|d\eta\leq \varepsilon\left(\|f_n\|_2^2+\|f\|_2^2\right) \leq 2\varepsilon m(B)^2\]
for all $n \geq 1$. Combining these results yields
\[\limsup_{n \to \infty}\left|\widehat{f_n(f_n * \sigma)}(\xi)-\widehat{f(f * \sigma)}(\xi)\right|\leq 2\varepsilon m(B)^2\]
and since $\varepsilon>0$ and $\xi \in\mathbb{R}^2$ are arbitrary we deduce that \eqref{bollll} holds for all $\xi \in \mathbb{R}^2$. By the dominated convergence theorem we obtain $\lim_{n \to \infty} \mathbf{D}_4(f_n)=\mathbf{D}_4(f)$ and hence for every sufficiently large $n$ we have $\mathbf{D}_1^\alpha(f_n)>0$ simultaneously for all $\alpha \in (0,M]$. The proof is complete.

\section{Acknowledgments}
The author would like to thank R. Nair, T. J. Sullivan and D. McCormick for helpful conversations.

\bibliographystyle{siam}

\bibliography{dist}

\end{document}